\documentclass[reqno,11pt,a4paper]{amsart}
\usepackage{graphics,epic}
\usepackage{amsmath,amssymb, amsthm}
\usepackage[all,2cell]{xy}
\usepackage{color}

\numberwithin{equation}{section}
\newtheorem{theorem}{Theorem}[section]
\newtheorem*{theorem*}{Theorem}

\newtheorem{lemma}[theorem]{Lemma}
\newtheorem{proposition}[theorem]{Proposition}
\newtheorem{corollary}[theorem]{Corollary}

\newtheorem*{conjecture*}{Conjecture}

\newcommand{\ie}{{\em i.e.}\ }

\newcommand{\eg}{{\em e.g.}\ }

\newcommand{\opname}[1]{\operatorname{\mathsf{#1}}}

\renewcommand{\mod}{\opname{mod}\nolimits}

\newcommand{\proj}{\opname{proj}\nolimits}

\newcommand{\Mod}{\opname{Mod}\nolimits}

\newcommand{\Grmod}{\opname{Grmod}\nolimits}
\newcommand{\per}{\opname{per}\nolimits}
\newcommand{\add}{\opname{add}\nolimits}
\newcommand{\Add}{\opname{Add}\nolimits}

\newcommand{\thick}{\opname{thick}\nolimits}

\newcommand{\im}{\opname{im}\nolimits}
\renewcommand{\ker}{\opname{ker}\nolimits}

\newcommand{\id}{\mathrm{id}}

\newcommand{\Hom}{\opname{Hom}}

\newcommand{\cHom}{{\mathcal{H}{\it om}}}
\newcommand{\cEnd}{\mathcal{E}{\it nd}}

\newcommand{\End}{\opname{End}}

\newcommand{\ten}{\otimes}

\newcommand{\ca}{{\mathcal A}}
\newcommand{\cb}{{\mathcal B}}
\newcommand{\cc}{{\mathcal C}}
\newcommand{\cd}{{\mathcal D}}

\newcommand{\cp}{{\mathcal P}}

\newcommand{\cs}{{\mathcal S}}

\newcommand{\cx}{{\mathcal X}}

\begin{document}

\title[From simple-minded collections to silting objects]{From simple-minded collections to silting objects via Koszul duality}

\author{Hao Su}
\address{Hao Su, School of Mathematical Sciences, University of Science and Technology of China, Hefei, Anhui 230026, P. R. China} \email{suhao@mail.ustc.edu.cn}

\author{Dong Yang}
\address{Dong Yang, Department of Mathematics, Nanjing University, Nanjing 210093, PR China}
\email{yangdong@nju.edu.cn}
\date{Last modified on \today}

\begin{abstract} Given an elementary simple-minded collection in the derived category of a non-positive dg algebra with finite-dimensional total cohomology, we construct a silting object via Koszul duality.
\end{abstract}

\maketitle


\section{Introduction}

Projective modules, simple modules and the interaction between them are crucial in the representation theory of finite-dimensional algebras. In the context of triangulated categories, silting objects and simple-minded collections are generalisations of projective modules and simple modules, respectively. Silting objects are `generators' of bounded co-t-structures (\cite{AiharaIyama12,Bondarko10,MendozaSaenzSantiagoSouto10}), and simple-minded collections are `generators' of  bounded t-structures whose hearts satisfy the Jordan--H\"older property (\cite{Al-Nofayee09,KoenigYang14}).

In \cite{Rickard02} (and \cite{RickardRouquier10,KoenigYang14}), Rickard provided a method to construct a silting object of the bounded homotopy category $K^b(\proj\Lambda)$ from a given simple-minded collection in the bounded derived category $\cd^b(\mod\Lambda)$, where $\Lambda$ is a finite-dimensional algebra. In this paper, we provide an alternative approach to Rickard's construction. Our approach works more generally for  non-positive dg algebras with finite-dimensional total cohomology.

\begin{theorem}\label{thm:main-thm-1}
Let $A$ be a non-positive dg algebra over a field $k$ with finite-dimensional total cohomology. For an elementary simple-minded collection $\{X_1,\ldots,X_r\}$ of the finite-dimensional derived category $\cd_{fd}(A)$, there exists a unique (up to isomorphism) silting object $M=M_1\oplus\ldots\oplus M_r$ of the perfect derived category $\per(A)$ such that
for $1\leq i,j\leq r$ and
$p\in\mathbb{Z}$
\[\Hom_{\cd_{fd}(A)}(M_i,\Sigma^p X_j)=\begin{cases} k & \text{if }i=j\text{ and } p=0,\\
                                           0 & \text{otherwise.}
                                           \end{cases}\]
\end{theorem}

The key idea of the construction of $M$ is to use the triangulated version of Koszul duality (\cite{Keller94,Lefevre03,LuPalmieriWuZhang08}). On the Koszul dual side we use $A_\infty$-algebras instead of dg algebras. One reason is that in this way the simple modules are easily obtained because the involved $A_\infty$-algebras are naturally augmented.

As a consequence of Theorem~\ref{thm:main-thm-1}, we obtain the following theorem, generalising \cite[Theorems 6.1 and 7.12]{KoenigYang14}.

\begin{theorem} Let $A$ be a 
non-positive dg algebra over an algebraically closed field with finite-dimensional total cohomology.
There are one-to-one correspondences which commute with mutations and which preserve partial orders
between
\begin{itemize}
 \item[(1)] equivalence classes of silting objects in $\per(A)$,
 \item[(2)] isomorphism classes of simple-minded collections in $\cd_{fd}(A)$,
 \item[(3)] bounded $t$-structures of $\cd_{fd}(A)$ with length heart,
 \item[(4)] bounded co-$t$-structures of $\per(A)$.
\end{itemize}
\end{theorem}

A similar result is claimed by Keller and Nicol\'as in an unpublished manuscript, with $A$ replaced by a homologically smooth non-positive dg algebra with finite-dimensional cohomology in each degree.

The paper is structured as follows. In Section~\ref{s:silting-and-non-positive-dg-alg}, we recall the basics on silting objects and non-positive dg algebras. In Section~\ref{s:A-infinite-alg}, we recall the basics on $A_\infty$-algebras and $A_\infty$-modules. In Section~\ref{s:smc-and-positive-A-infinite-alg}, we recall the definition of simple-minded collections and study strictly unital minimal positive $A_\infty$-algebras, which are closely related to simple-minded collections. In Section~\ref{s:construction} we provide the construction of a silting object from a given simple-minded collection. 

Throughout the paper, let $k$ be a field and let $D=\Hom_k(?,k)$ be the $k$-dual.  Without further remark, modules will be right modules and all categories are $k$-categories.
\medskip

\noindent{Acknowledgement.} The first-named author is deeply indebted to the second-named author for his guidance and help. The second-named author would like to thank Mark Blume for valuable remarks on an earlier version. He acknowledges support from the National Science Foundation in China No. 11401297 and a project funded by the Priority Academic Program Development of Jiangsu Higher Education Institutions.

\section{Silting objects and non-positive dg algebras}\label{s:silting-and-non-positive-dg-alg}

In this section we recall the definitions and standard facts on silting objects and non-positive dg algebras.

\subsection{Silting objects}
Let $\cc$ be a triangulated category with suspension functor $\Sigma$. For a set $\cs$ of objects of $\cc$, let $\add(\cs)=\add_\cc(\cs)$ be the smallest full subcategory of $\cc$ containing $\cs$ and closed under taking direct summands and finite direct sums, and let $\thick(\cs)=\thick_\cc(\cs)$ denote the thick subcategory of $\cc$ generated by $\cs$, \ie the smallest full subcategory of $\cc$ that is closed under taking suspensions, cosuspensions, extensions and direct summands.

An object $M$ of $\cc$ is a \emph{silting object} of $\cc$ if
\begin{itemize}
\item[-]
$\Hom_{\cc}(M,\Sigma^p M)=0$ for any $p>0$,
\item[-]
$\cc=\thick(M)$.
\end{itemize}
Two silting objects $M$ and $M'$ of $\cc$ are said to be \emph{equivalent} if $\add(M)=\add(M')$.

\subsection{Non-positive dg algebras}

 Let $A$ be a dg $k$-algebra. Denote by $\cd(A)$ the derived category of (right) dg $A$-modules (see \cite{Keller94,Keller06d}), which is a triangulated category with suspension functor $\Sigma$ the shift functor. Let $\per(A)=\thick(A_A)$, the thick subcategory of $\cd(A)$ generated by $A_A$, the free dg $A$-module of rank $1$, and let $\cd_{fd}(A)$ be the full subcategory of $\cd(A)$ consisting of dg $A$-modules whose total cohomology is finite-dimensional. If $A$ is a finite-dimensional $k$-algebra, then a dg $A$-module is exactly a complex of $A$-modules. So $\cd(A)=\cd(\Mod A)$ and we have canonical triangle equivalences $K^b(\proj A)\to \per(A)$ and $\cd^b(\mod A)\to \cd_{fd}(A)$.
For two dg $A$-modules $M$ and $N$,  let $\cHom_A(M,N)$ denote the
complex whose degree $p$ component consists of those $A$-linear
maps from $M$ to $N$ which are homogeneous of degree $p$, and whose
differential  takes $f$ to \mbox{$d_N\circ
f-(-1)^{|f|} f\circ d_M$,} where $f$ is homogeneous of degree $|f|$.

A dg $A$-module $M$ is said to be \emph{$K$-projective} if $\cHom_A(M,N)$ is acyclic for any dg $A$-module $N$ which is acyclic. Let $M$ and $N$ be dg $A$-modules. If $M$ is $K$-projective, then there is a canonical isomorphism
\[
\Hom_{\cd(A)}(M,N)\cong H^0\cHom_A(M,N).
\]
By \cite[Theorem 3.1]{Keller94}, for any dg $A$-module $M$ there is a quasi-isomorphism $\mathbf{p}M\to M$ of dg $A$-modules with $\mathbf{p}M$ being $K$-projective.

\smallskip

We say that $A$ is \emph{non-positive} if $A^p=0$ vanishes for all $p>0$.
A triangulated category is said to be \emph{algebraic} if it is triangle equivalent to the stable category of a Frobenius category.

\begin{lemma}[{\cite[Lemma 4.1]{KoenigYang14}}]\label{l:negative-dg-and-silting}
\begin{itemize}
\item[(a)]
 Let $A$ be a non-positive dg $k$-algebra. The free dg $A$-module of rank $1$ is a silting object of $\per(A)$.
\item[(b{)}]
 Let $\cc$ be an idempotent complete algebraic triangulated category
and let $M\in\cc$ be a silting object. Then there is a non-positive dg
$k$-algebra $A$ together with a triangle equivalence
$\per(A)\stackrel{\simeq}{\longrightarrow} \cc$ which takes $A$ to $M$.
\end{itemize}
\end{lemma}

\subsection{Cohomologically finite-dimensional non-positive dg algebras}
\label{ss:non-positive-dg-alg}

Let $A$ be a non-positive dg $k$-algebra whose total cohomology is finite-dimensional over $k$. Then both $\cd_{fd}(A)$ and $\per(A)$ are Krull--Schmidt, since they are Hom-finite, see for example \cite[Proposition 4.2 (a)]{AdachiMizunoYang17}. Moreover,  $\per(A)\subseteq \cd_{fd}(A)$ and $\thick(D({}_AA))\subseteq\cd_{fd}(A)$. There is a triangle functor $\nu: \cd(A)\to\cd(A)$ (called the \emph{Nakayama functor}) which restricts to a triangle equivalence $\nu:\per(A)\to\thick(D({}_AA))$. We have the Auslander--Reiten formula
\[D\Hom(M,N)\cong \Hom(N,\nu(M))
\]
for $M\in\per(A)$ and $N\in\cd(A)$. See \cite[Section 10]{Keller94}.

Recall from the preceding subsection that $A_A$ is a silting object of $\per(A)$. Let $M$ be a silting object of $\per(A)$. We may assume that $M$ is $K$-projective and form the dg endomorphism algebra $\cEnd_A(M):=\cHom_A(M,M)$. Then by \cite[Lemma 6.1]{Keller94}, we have a derived equivalence $\cd(\cEnd_A(M))\to\cd(A)$ taking $\cEnd_A(M)$ to $M$. Since $H^p\cEnd_A(M)=\Hom_{\cd(A)}(M,\Sigma^p M)$, it follows that $\cEnd_A(M)$ has finite-dimensional total cohomology and $H^p\cEnd_A(M)=0$ for all $p>0$. The subcomplex $\sigma^{\leq 0}\cEnd_A(M)$, where $\sigma^{\leq 0}$ is the standard truncation in degree $0$, is a dg subalgebra of $\cEnd_A(M)$. In particular, it is a non-positive dg algebra with finite-dimensional total cohomology. Moreover, the canonical embedding $\sigma^{\leq 0}\cEnd_A(M)\to \cEnd_A(M)$ is a quasi-isomorphism of dg algebras, inducing a derived equivalence
\[
\cd(\sigma^{\leq 0}\cEnd_A(M))\to \cd(A)
\]
taking $\sigma^{\leq 0}\cEnd_A(M)$ to $M$. We call $\sigma^{\leq 0}\cEnd_A(M)$ the \emph{truncated dg endomorphism algebra} of $M$.

\section{$A_\infty$-algebras and $A_\infty$-modules}\label{s:A-infinite-alg}

In this section we recall the definition and basic
properties of $A_\infty$-algebras and $A_\infty$-modules.
We follow~\cite{Lefevre03} and also refer to \cite{Keller01} and
\cite{LuPalmieriWuZhang08}.

\subsection{$A_\infty$-algebras}
Let $R$ be a separable semi-simple $k$-algebra. An \emph{$A_\infty$-algebra} $A$ over $R$ is a graded $R$-bimodule endowed with a family of homogeneous $R$-bilinear maps $m_n:A^{\ten_R
n}\rightarrow A ~~(n\geq 1)$ of degree $2-n$, called the \emph{multiplications} of $A$, which satisfy the following identities
\begin{eqnarray}\label{e:associativity-of-multiplication}
\sum_{i+j+l=n}(-1)^{ij+l}m_{i+1+l}(\id^{\ten i}\ten m_j\ten \id^{\ten l})=0, n\geq 1.
\end{eqnarray}
Here $i,l\geq 0$ and $j\geq 1$. (In \cite[D\'efinition 1.2.1.1]{Lefevre03}, $A_\infty$-algebras are defined over a monoidal category $\mathbf{C}$. The above definition is obtained by taking $\mathbf{C}$ as the category of $R$-bimodules.)
We are mainly interested in the case when $R$ is a finite direct product of copies of $k$, which we use `as if it were non-commutative', namely, we do not require that the left and right graded $R$-module structures on $A$ coincide (compare \eg \cite[Section 2.1]{Lunts10}).  Let $A$ be an $A_\infty$-algebra over $R$. $A$ is said to be \emph{strictly unital} if there is a $R$-bilinear map $\eta: R\to A$ (called the \emph{unit} of $A$) which is homogeneous of degree $0$ such that $m_n(\id\ten\cdots\ten \id\ten \eta\ten \id\ten \cdots\ten \id)=0$ for $n\neq 2$ (here $\eta$ can occur in any position in the tensor product) and $m_2(\id\ten\eta)=m_2(\eta\ten \id)=\id$.
Note that the identity \eqref{e:associativity-of-multiplication} for $n=1$ is $m_1^2=0$, thus $A$ is a complex of $R$-bimodules  with differential $m_1$. $A$ is said to be \emph{minimal} if $m_1=0$. In this case, $A$ is a graded algebra over $R$ with $m_2$ as multiplication.

Let $A$ and $B$ be two strictly unital $A_\infty$-algebras  over $R$. A \emph{strictly unital $A_\infty$-morphism} $f:A\rightarrow B$ of strictly unital $A_\infty$-algebras is a family of homogeneous $R$-bilinear maps $f_n: A^{\ten_R n}\rightarrow B ~~(n\geq 1)$ of degree $1-n$,  such that $f_1\eta_A=\eta_B$, $f_n(\id\ten\cdots \ten \id\ten\eta_A\ten \id\ten \cdots\ten \id)=0$ for all $n\geq 2$, and that
\begin{eqnarray}\label{e:associativity-of-morphisms}
     \sum_{i+j+l=n}(-1)^{ij+l}f_{i+1+l}(\id^{\otimes i} \otimes m_{j} \otimes\id^{\otimes l})
   = \sum_{\substack{1\leq p \leq n \\ i_{1}+\ldots+i_{p}=n}} (-1)^{\omega}m_{p}(f_{i_{1}} \otimes \cdots \otimes f_{i_{p}}),n \geq 1.
\end{eqnarray}
Here $j\geq 1;i,l \geq 0$ and $\omega = \sum_{2\leq u\leq p}(1-i_u)\sum_{1\leq v\leq u}i_v$.
It follows that $f_1$ is a chain map with respect to the differentials $m_1$. If $f_1$ is a quasi-isomorphism of complexes, we say that $f$ is an   \emph{$A_\infty$-quasi-isomorphism}.
If $f_n=0$ for $n\geq 2$, then the above identities amounts to saying that $f_1:A\rightarrow B$ commutes with all multiplications $m_n$. In this case, we say that $f$ is \emph{strict} and identify $f$ with $f_1$.

Let $A$ be a strictly unital $A_\infty$-algebra over $R$. $A$ is said to be \emph{augmented} if there is a strict $A_\infty$-morphism $\varepsilon: A\to R$ of strictly unital $A_\infty$-algebras, which is called the \emph{augmentation} of $A$. Here we view $R$ as a strictly unital $A_\infty$-algebra over $R$ with $m_2=\id$, $m_n=0$ for $n\neq 2$ and $\eta=\id$. Let $A$ and $B$ be two augmented $A_\infty$-algebras over $R$. An \emph{$A_\infty$-morphism} $f: A\to B$ of augmented $A_\infty$-algebras is a strictly unital $A_\infty$-morphism of strictly unital $A_\infty$-algebras such that $\varepsilon_B f_1=\varepsilon_A$.

\subsection{Enveloping dg algebras}

A dg algebra $A$ over $R$ is a dg $k$-algebra together with a homomorphism $\eta:R\to A$ of dg $k$-algebras.  It can be considered as a strictly unital $A_\infty$-algebra over $R$ with $m_1$ being the differential, $m_2$ being the multiplication and $m_n=0$ for $n\geq 3$.

\begin{theorem}[{\cite[Proposition 7.5.0.2]{Lefevre03} and \cite[Lemme 2.3.4.3]{Lefevre03}}]\label{t:enveloping-algebra}
Let $A$ be a strictly unital $A_\infty$-algebra over $R$. Then there is a dg algebra $A'$ (called a \emph{dg model} of $A$) over $R$ with a strictly unital $A_\infty$-quasi-isomorphism $A \to A'$. If $A$ is augmented over $R$, then $A'$ can be taken augmented over $R$ (in this case $A'$ is called the \emph{enveloping dg algebra of $A$}) and the $A_\infty$-quasi-isomorphism above is an $A_\infty$-quasi-isomorphism of augmented $A_\infty$-algebras over $R$.
\end{theorem}

\subsection{$A_\infty$-modules}
Let $A$ be a strictly unital $A_\infty$-algebra over $R$. A \emph{(right)
$A_\infty$-module} over $A$ is a graded right $R$-module $M$ endowed
with a family of homogeneous $R$-linear maps $m_n^M:M\ten_R A^{\ten_R n-1}\longrightarrow M~~(n\geq 1)$
of degree $2-n$ such that (some of the superscripts $M$ on $m$ are
omitted)
\begin{eqnarray}\label{e:associativity-for-modules}
\sum_{i+j+l=n}(-1)^{ij+l}m_{i+1+l}(\id^{\ten i}\ten m_j\ten \id^{\ten l})=0.
\end{eqnarray}
Here $i,l\geq 0$ and $j\geq 1$. The graded $R$-module $M$ equipped
with $m_1^M$ becomes a complex. An $A_\infty$-module $M$ is said to
be \emph{minimal} if $m_1^M=0$. It is said to be \emph{strictly unital}
if $m_n^M(\id_M\ten \id\ten
\cdots\ten\id\ten \eta\ten \id\ten\cdots\ten \id)=0$ for all $n\geq 3$, and $m_2^M(\id_M\ten \eta)=\id_M$. If $M'$ is a graded $R$-submodule of $M$ such that $m^{M}_{n}$ restricts to $M'$ for all $n \geq 1$, then $M'$ together with the restriction of $m_n^M$ is called a \emph{submodule} of $M$. $A$ together with its multiplications is an $A_\infty$-module over $A$. An element $e$ of $A$ is called a \emph{strict idempotent} if $e\in\im(\eta)$, $m_2(e\ten e)=e$ and for all $n\neq 2$ we have $m_n(a_1\ten\cdots\ten a_n)=0$ if one of $a_1,\ldots,a_n$ is $e$. If $e$ is a strict idempotent of $A$, then $eA=\{ea:=m_2(e\ten a)|a\in A\}$ is an $A_\infty$-submodule of $A$, because
\[m_n(e a_1\ten a_2\ten\cdots\ten a_n)=e m_n(a_1\ten a_2\ten\cdots\ten a_n).\]

Let $M$ and $M'$ be two strictly unital $A_\infty$-modules over $A$.
A \emph{strictly unital $A_\infty$-morphism} $f:M\rightarrow M'$ is a family of
homogeneous $R$-linear maps $f_n:M\ten_R A^{\ten_R n-1}\longrightarrow M' ~~(n\geq
1)$ of degree $1-n$ such that $f_n(\id_M\ten\id\ten\cdots\ten\id\ten\eta\ten\id\ten\cdots\ten\id)=0$ for all $n\geq 2$ and that  the following identity holds for all $n\geq
1$
\begin{eqnarray}\label{e:associativity-for-morphisms-of-modules}
\sum_{i+j+l=n}(-1)^{ij+l}f_{i+1+l}(\id^{\ten i}\ten m_j\ten \id^{\ten l})=\sum_{s+t=n} m_{1+t}(f_s\ten \id^{\ten
t}).
\end{eqnarray}
Here $i,l,t\geq 0$ and $j,s\geq 1$. In particular, $f_1$ is a chain
map of complexes. $f$ is an
\emph{$A_\infty$-quasi-isomorphism} if $f_1$ induces identities on all
cohomologies. $f$ is \emph{strict} if $f_n=0$ for all $n\geq
2$. We will identify $f$ with $f_1$ in this case.

\begin{proposition}\label{p:minimal-model}\emph{(\cite[Proposition 3.3.1.7]{Lefevre03})}
Let $A$ be a strictly unital $A_\infty$-algebra over $R$ and $M$ be a strictly unital $A_\infty$-module over $A$. Then there is a strictly
unital minimal $A_\infty$-module over $A$ which is $A_\infty$-quasi-isomorphic to
$M$.
\end{proposition}

\subsection{Derived categories}
Let $A$ be a strictly unital $A_\infty$-algebra over $R$. Let $\Mod_{\infty}(A)$ be the category of strictly unital
$A_\infty$-modules over $A$ with strictly unital
$A_\infty$-morphisms as morphisms. The \emph{derived category}
$\cd(A)$ is the category obtained from $\Mod_{\infty}(A)$ by
formally inverting all $A_\infty$-quasi-isomorphisms. The category $\cd(A)$ is
a triangulated category whose suspension functor is the shift
functor $\Sigma$.  It has arbitrary (set-indexed) direct sums.

Let $A$ be a dg algebra over $R$. There are three classes of modules we can consider. First we can view $A$ as a strictly unital $A_\infty$-algebra over $R$ with vanishing $m_n$ for $n\geq 3$ and consider strictly unital $A_\infty$-modules over $A$. Secondly, we can consider unital dg $A$-modules (see \cite[Section 2.1.1]{Lefevre03}), which are exactly the strictly unital $A_\infty$-modules $M$ over $A$ with $m_n^M=0$ for $n\geq 3$. Thirdly,  we can consider dg modules over $A$ which is considered a dg $k$-algebra. It is easy to check that the second and the third classes coincide. By \cite[Lemme 4.1.3.8]{Lefevre03}, the derived category of dg $A$-modules is canonically equivalent to the derived category of strictly unital $A_\infty$-modules over $A$. We will identify these two derived categories.

\begin{theorem}\label{t:quasi-isomorphism-induce-derived-equivalence} \emph{({\cite[Th\'eor\`eme 4.1.2.4]{Lefevre03}})}
  Let $f:A\to B$ be a strictly unital $A_\infty$-quasi-isomorphism of strictly unital $A_\infty$-algebras over $R$. Then there is a triangle equivalence $f^*:\cd(B)\to\cd(A)$ which takes $B_B$ to an $A_\infty$-module isomorphic to $A_A$ in $\cd(A)$.
\end{theorem}

\subsection{Morphisms from `projectives'}
Let $A$ be a strictly unital $A_\infty$-algebra over $R$. Let $e$ be a strict idempotent of $A$. For a strictly unital $A_\infty$-module $M$ over $A$ let $Me=\{me:=m_2^M(m\ten e)|m\in M\}$. Then applying the identity \eqref{e:associativity-for-modules} for $n=2$ to $m\ten e$, we get
\[
m_1(me)=(-1)^{|m|} mm_1(e)+m_1(m)e=m_1(m)e\in Me.
\]
So $Me$ is a subcomplex of $M$. We will need the following result.

\begin{lemma}\label{l:morphism-from-projective}
Let $e$ be a strict idempotent of $A$. For a strictly unital $A_\infty$-module $M$ over $A$ and an
integer $p$, 
there is an isomorphism
\begin{eqnarray}\label{eq:morphism-space-from-free-module}
\Hom_{\cd(A)}(eA,\Sigma^p M)\cong H^p(Me).
\end{eqnarray}
\end{lemma}
\begin{proof}
This can be obtained as a consequence of a suitable version of Yoneda's lemma. We do not find a reference in the literature, so we give a direct proof here. It is enough to prove for the case $p=0$.
By \cite[Th\'eor\`eme 4.1.3.1]{Lefevre03}, $\Hom_{\cd(A)}(eA, M)$ is the same as the space of strictly unital $A_\infty$-morphisms from $eA$ to $M$ modulo those homotopic to $0$. We show in four steps that this space is canonically isomorphic to $H^0(Me)$. Here two strictly unital $A_\infty$-morphisms $f,g: M\to M'$ of strictly unital $A_\infty$-modules are homotopic if there is a strictly unital homotopy $h$ between $f$ and $g$, that is, a family of homogeneous $K$-linear maps $h_n:M\ten_R A^{\ten_R(n-1)}\to M'$ ($n\geq 1$) of degree $-n$ such that $h_n(\id_M\ten \id\ten\cdots\ten \id\ten \eta\ten\id\ten\cdots\ten \id)=0$ for all $n\geq 2$ and that the following identity holds for all $n\geq 1$
\begin{eqnarray}\label{e:homotopy}
f_n-g_n=\sum_{s+t=n}(-1)^t m_{1+t}(h_s\ten \id^{\ten t})+\sum_{i+j+l=n}(-1)^{ij+l}h_{i+1+l}(\id^{\ten i}\ten m_j \ten \id^{\ten l}).
\end{eqnarray}

\smallskip
Step 1: Let $m\in Z^0(Me)$. For $n\geq 1$, define a homogeneous $R$-linear map of degree $1-n$
\[
\xymatrix@R=0.5pc{f_n: eA\ten_R A^{\ten_R (n-1)}\ar[r] & M\\
a_1\ten a_2\ten\cdots\ten a_n\ar@{|->}[r] & (-1)^{n+1}m_{n+1}(m\ten a_1\ten\cdots\ten a_n).
}
\]
Then $f=(f_n)_{n\geq 1}$ is a strictly unital $A_\infty$-morphism from $eA$ to $M$.

The identity $f_n(\id_{eA}\ten\id\ten\cdots\ten\id\ten\eta\ten\id\ten\cdots\ten\id)=0$ ($n\geq 2$) is clear since $A$ is strictly unital. We need to check the identity \eqref{e:associativity-for-morphisms-of-modules} for all $n\geq 1$ applied to $a_1\ten a_2\ten \cdots a_n$, where $a_1\in eA$ and $a_2,\ldots,a_n\in A$ are homogenous. We have (note that when flipping tensors we have the Koszul sign: $(\varphi\ten \psi)(u\ten v)=(-1)^{|\psi|\cdot|u|}(\varphi(u)\ten\psi(v))$)
\begin{align*}
\mathrm{LHS}&=\sum_{i+j+l=n}(-1)^{ij+l}f_{i+1+l}((-1)^{(|a_1|+\ldots+|a_i|)(2-j)}a_1\ten\cdots\ten a_i\\
&\hspace{50pt}\ten m_j(a_{i+1}\ten\cdots\ten a_{i+j})\ten a_{i+j+1}\ten\cdots\ten a_n)\\
&=\sum_{i+j+l=n}(-1)^{ij+l}(-1)^{i+2+l}m_{i+2+l}((-1)^{(|a_1|+\ldots+|a_i|)(2-j)}m\ten a_1\ten\cdots\ten a_i\\
&\hspace{50pt}\ten m_j(a_{i+1}\ten\cdots\ten a_{i+j})\ten a_{i+j+1}\ten\cdots\ten a_n)\\
&=\sum_{i+j+l=n+1, i\geq 1}(-1)^{ij+l+n}m_{i+1+l}(\id^{\ten i}\ten m_j\ten \id^{\ten l})(m\ten a_1\ten\cdots\ten a_n)\\[5pt]
\mathrm{RHS}&=\sum_{s+t=n}m_{1+t}(f_s(a_1\ten\cdots\ten a_s)\ten a_{s+1}\ten\cdots\ten a_n)\\
&=\sum_{s+t=n}(-1)^{s+1} m_{1+t}(m_{s+1}(m\ten a_1\ten\cdots\ten a_s)\ten a_{s+1}\ten\cdots\ten a_n)\\
&=\sum_{s+t=n+1,s\geq 2}(-1)^{n+1+t}m_{1+t}(m_s\ten \id^{\ten t})(m\ten a_1\ten\cdots\ten a_n).
\end{align*}
By \eqref{e:associativity-for-modules}, we have
\begin{align*}
\mathrm{LHS}-\mathrm{RHS}&=-m_{1+n}(m_1\ten \id^{\ten n})(m\ten a_1\ten\cdots\ten a_n)\\
&=-m_{1+n}(m_1(m)\ten a_1\ten\cdots\ten a_n)\\
&=0.
\end{align*}

\smallskip
Step 2: Let $m\in Z^0(Me)$. In Step 1, we associate to $m$ a strictly unital $A_\infty$-morphism $f$ from $eA$ to $M$. We claim that $f$ is homotopic to $0$ if and only if $m$ belongs to $B^0(Me)$.

We first show that for any $m\in Me$ we have $m=me$. Suppose $m=m'e$. Applying \eqref{e:associativity-for-modules} for $n=3$ to $m'\ten e\ten e$ we get
\[m=m'e=m'(ee)=(m'e)e=me.\]

Now we prove the `only if' part.  Assume that $f$ is homotopic to $0$. Then there exists a homogeneous $R$-linear map $h_1:eA\to M$ of degree $-1$ such that $f_1=m_1h_1+h_1m_1$. So
\[
m=me=f_1(e)=m_1h_1(e)+h_1m_1(e)=m_1h_1(e)\in B^0(Me),
\]
as $e$ is a strict idempotent.

Next we prove the `if' part. Assume that $m=m_1(m')$, where $m'\in M^{-1}$. For $n\geq 1$, define a homogeneous $R$-linear map of degree $-n$
\[
\xymatrix@R=0.5pc{ h_n: eA\ten_R A^{\ten_R(n-1)}\ar[r] & M\\
a_1\ten a_2\ten\cdots\ten a_n\ar@{|->}[r] & m_{n+1}(m'\ten a_1\ten\cdots\ten a_n).
}
\]
Then $h=(h_n)_{\geq 1}$ is a strictly unital homotopy between $f$ and $0$. The identity $h_n(\id_{eA}\ten \id\ten\cdots\ten \id\ten \eta\ten\id\ten\cdots\ten \id)=0$ ($n\geq 2$) is clear since $A$ is strictly unital. We need to check the identity \eqref{e:homotopy} for all $n\geq 1$ applied to $a_1\ten a_2\ten\cdots\ten a_n$, where $a_1\in eA$ and $a_2,\ldots,a_n\in A$ are homogeneous. We have
\begin{align*}
\mathrm{RHS}&=\sum_{s+t=n}(-1)^t m_{1+t}(h_s(a_1\ten\cdots \ten a_s)\ten a_{s+1}\ten\cdots\ten a_n)\\
&\hspace{10pt}+\sum_{i+j+l=n}(-1)^{ij+l} h_{i+1+l}((-1)^{(|a_1|+\ldots+|a_i|)(2-j)}a_1\ten\cdots\ten a_i\\
&\hspace{60pt}\ten m_j(a_{i+1}\ten\cdots\ten a_{i+j})\ten a_{i+j+1}\ten\cdots\ten a_n)\\
&=\sum_{s+t=n}(-1)^t m_{1+t}(m_{s+1}(m'\ten a_1\ten\cdots \ten a_s)\ten a_{s+1}\ten\cdots\ten a_n)\\
&\hspace{10pt}+\sum_{i+j+l=n}(-1)^{(i+1)j+l} m_{i+2+l}((-1)^{(|m'|+|a_1|+\ldots+|a_i|)(2-j)}m'\ten a_1\ten\cdots\ten a_i\\
&\hspace{60pt}\ten m_j(a_{i+1}\ten\cdots\ten a_{i+j})\ten a_{i+j+1}\ten\cdots\ten a_n)\\
&=\sum_{s+t=n+1,s\geq 2} (-1)^t m_{1+t}(m_s\ten \id^{\ten t}))(m'\ten a_1\ten\cdots\ten a_n)\\
&\hspace{10pt}+\sum_{i+j+l=n+1,i\geq 1}(-1)^{ij+l}m_{i+1+l}(\id^{\ten i}\ten m_j\ten \id^{\ten l})(m'\ten a_1\ten\cdots\ten a_n)\\
&\hspace{-4pt}\stackrel{\eqref{e:associativity-for-modules}}{=}-(-1)^n m_{1+n}(m_1\ten \id^{\ten n})(m'\ten a_1\ten \cdots\ten a_n)\\
&=(-1)^{n+1} m_{n+1}(m\ten a_1\ten\cdots\ten a_n)\\
&=f_n(a_1\ten\cdots\ten a_n) =\mathrm{LHS}.
\end{align*}

\smallskip
Step 3:  Let $f$ be a strictly unital $A_\infty$-morphism from $eA$ to $M$.
We first show that $f_1(e)\in Z^0(Me)$. The identity \eqref{e:associativity-for-morphisms-of-modules} for $n=1$ applied to $e$ yields $m_1(f_1(e))=f_1(m_1(e))=0$, so $f_1(e)\in Z^0(M)$. The same identity for $n=2$ applied to $e\ten e$ yields $f_1(e)=f_1(e)e+m_1f_2(e\ten e)=f_1(e)e\in Me$. The last equality holds because $f_2(\id\ten\eta)=0$ and $e\in\im(\eta)$.

We claim that $f$ is homotopic to the  strictly unital $A_\infty$-morphism from $eA$ to $M$ associated to $f_1(e)$ as in Step 1. It is enough to show that if $f_1(e)=0$, then $f$ is homotopic to $0$.

For $n\geq 1$, define a homogeneous $R$-linear map of degree $-n$
\[
\xymatrix@R=0.5pc{ h_n: eA\ten_R A^{\ten_R(n-1)}\ar[r] & M\\
a_1\ten a_2\ten\cdots\ten a_n\ar@{|->}[r] & (-1)^{n+1}f_{n+1}(e\ten a_1\ten\cdots\ten a_n).
}
\]
Then $h=(h_n)_{n\geq 1}$ is a strictly unital homotopy between $f$ and $0$. The identity $h_n(\id_{eA}\ten \id\ten\cdots\ten \id\ten \eta\ten\id\ten\cdots\ten \id)=0$ ($n\geq 2$) is clear since $f$ is strictly unital. We need to check the identity \eqref{e:homotopy} for all $n\geq 1$ applied to $a_1\ten a_2\ten\cdots\ten a_n$, where $a_1\in eA$ and $a_2,\ldots,a_n\in A$ are homogeneous. We have
\begin{align*}
\mathrm{RHS}&=\sum_{s+t=n}(-1)^t m_{1+t}(h_s(a_1\ten\cdots \ten a_s)\ten a_{s+1}\ten\cdots\ten a_n)\\
&\hspace{10pt}+\sum_{i+j+l=n}(-1)^{ij+l} h_{i+1+l}((-1)^{(|a_1|+\ldots+|a_i|)(2-j)}a_1\ten\cdots\ten a_i\\
&\hspace{60pt}\ten m_j(a_{i+1}\ten\cdots\ten a_{i+j})\ten a_{i+j+1}\ten\cdots\ten a_n)\\
&=\sum_{s+t=n}(-1)^t m_{1+t}((-1)^{s+1}f_{s+1}(e\ten a_1\ten\cdots \ten a_s)\ten a_{s+1}\ten\cdots\ten a_n)\\
&\hspace{10pt}+\sum_{i+j+l=n}(-1)^{ij+l}(-1)^{i+2+l} f_{i+2+l}((-1)^{(|a_1|+\ldots+|a_i|)(2-j)}e\ten a_1\ten\cdots\ten a_i\\
&\hspace{60pt}\ten m_j(a_{i+1}\ten\cdots\ten a_{i+j})\ten a_{i+j+1}\ten\cdots\ten a_n)\\
&=\sum_{s+t=n+1,s\geq 2} (-1)^{n+1}m_{1+t}(f_s\ten \id^{\ten t}))(e\ten a_1\ten\cdots\ten a_n)\\
&\hspace{10pt}+\sum_{i+j+l=n}(-1)^{(i+1)j+l+n} f_{i+2+l}(\id^{\ten (i+1)}\ten m_j\ten \id^{\ten l})(e\ten a_1\ten\cdots\ten a_n)\\
&\hspace{-11pt}\stackrel{f_1(e)=0}{=}\sum_{s+t=n+1} (-1)^{n+1}m_{1+t}(f_s\ten \id^{\ten t}))(e\ten a_1\ten\cdots\ten a_n)\\
&\hspace{10pt}+\sum_{i+j+l=n+1,i\geq 1}(-1)^{ij+l+n} f_{i+1+l}(\id^{\ten i}\ten m_j\ten \id^{\ten l})(e\ten a_1\ten\cdots\ten a_n)\\
&\hspace{-4pt}\stackrel{\eqref{e:associativity-for-morphisms-of-modules}}{=}-(-1)^n \sum_{j+l=n+1}(-1)^l f_{1+l}(m_j\ten \id^{\ten l})(e\ten a_1\ten \cdots\ten a_n)\\
&=f_n(m_2\ten \id^{\ten n})(e\ten a_1\ten\cdots\ten a_n)\\
&=f_n(a_1\ten\cdots\ten a_n) =\mathrm{LHS}.
\end{align*}

\smallskip
Step 4: By Step 3, $\Hom_{\cd(A)}(eA, M)$ is canonically isomorphic to the space of strictly unital $A_\infty$-morphisms from $eA$ to $M$ associated to $m\in Z^0(Me)$ modulo those homotopic to $0$, and hence to $H^0(Me)=Z^0(Me)/B^0(Me)$ by Step 2.
\end{proof}

\subsection{Perfect derived categories and finite-dimensional derived categories}
Denote by $\per(A)$ the thick subcategory of $\cd(A)$
generated by $A_A$, and denote by $\cd_{fd}(A)$
the full subcategory of $\cd(A)$ consisting of those
$A_\infty$-modules whose total cohomology is finite-dimensional.

  \begin{lemma}\label{l:der-equiv-restricts-to-per-and-fd}
\begin{itemize}
\item[(a)] Let $M$ be an object of $\cd(A)$. Then $M$ is compact if and only if it belongs to $\per(A)$, and $M$ belongs to $\cd_{fd}(A)$ if and only if $\bigoplus_{i\in\mathbb{Z}} \Hom_{\cd(A)}(N,\Sigma^{i}M)$ is finite-dimensional for any $N\in\per(A)$.
\item[(b)]
    Let $A$ and $B$ be two strictly unital $A_\infty$-algebras over $R$. A triangle equivalence $\cd(A)\rightarrow \cd(B)$ restricts to triangle equivalences $\per(A)\rightarrow\per(B)$ and $\cd_{fd}(A)\rightarrow \cd_{fd}(B)$.
  \end{itemize}
  \end{lemma}

\begin{proof}
Let $A'$ be a dg model of $A$ as in Theorem~\ref{t:enveloping-algebra}. Then by Theorem~\ref{t:quasi-isomorphism-induce-derived-equivalence} there is a triangle equivalence $\cd(A')\to\cd(A)$ which takes $A'_{A'}$ to $A_A$.  By \cite[Corollary 3.7]{Keller06d} (also \cite[Remark 5.3 (a)]{Keller94}), an object of $\cd(A')$ is compact if and only if it belongs to $\per(A')$. The first assertion of (a) follows immediately. The second assertion of (a) follows from Lemma~\ref{l:morphism-from-projective} by d\'evissage. (b) is a direct consequence of (a) since $\per(A)$ and $\cd_{fd}(A)$ admit intrinsic descriptions inside $\cd(A)$.
  \end{proof}

\subsection{Morita theorems}
We have the following Morita theorems for derived categories and for algebraic triangulated categories. They can be considered as special cases of \cite[Th\'eor\`eme 7.6.0.4]{Lefevre03} and \cite[Th\'eor\`eme 7.6.0.6]{Lefevre03}, respectively.

  \begin{theorem}\label{t:morita-for-der-cat}
  Let $A$ be a strictly unital $A_\infty$-algebra over $R$. Let $\{X_1,\ldots,X_r\}$ be a set of compact generators of $\cd(A)$, \ie $\per(A)=\thick(X_1,\ldots,X_r)$. Let $K$ be the direct product of $r$ copies of $k$. Then there is a strictly unital minimal $A_\infty$-algebra $B$ over $K$ such that as a graded algebra
  \[B= \bigoplus_{p\in\mathbb{Z}}\Hom_{\cd(A)}(\bigoplus_{i=1}^r X_i,\Sigma^p \bigoplus_{i=1}^r X_i),
  \]
  and there is a triangle equivalence
    \[\cd(B)\longrightarrow \cd(A)\]
    taking  $e_iB$ ($1\leq i\leq r$) to $X_{i}$.
  \end{theorem}

\begin{proof}
By \cite[Th\'eor\`eme 7.6.0.4]{Lefevre03}, there is a strictly unital minimal $A_\infty$-category $\cb$ whose objects are $X_1,\ldots,X_r$, and whose morphism spaces are $\Hom_\cb(X_i,X_j)= \bigoplus_{p\in\mathbb{Z}}\Hom_{\cd(\ca)}(X_i,\Sigma^p X_j)$ for $1\leq i,j\leq r$, and a triangle equivalence
    \[\cd(\mathcal{B})\longrightarrow \cd(A),\]
    which takes the free $A_\infty$-module $X_i^\wedge$ associated to the object $X_i\in\mathcal{B}$  to $X_i\in\cd(A)$ for $1\leq i\leq r$.
We identify $K$ with $k\{\id_{X_1}\}\times\cdots\times k\{\id_{X_r}\}$. Let $B$ be the graded $K$-bimodule $\bigoplus_{i,j=1}^r \bigoplus_{p\in\mathbb{Z}}\Hom_{\cd(\ca)}(X_i,\Sigma^p X_j)$.  Then the multiplications on $\cb$ induce multiplications on $B$ such that $B$ becomes a minimal strictly unital $A_\infty$-algebra over $K$. The assignment $M\mapsto \bigoplus_{i=1}^r M(X_i)$ extends to an isomorphism $\Mod_\infty(\cb)\to\Mod_\infty(B)$, which induces a triangle isomorphism $\cd(\cb)\to \cd(B)$. So we have the desired triangle equivalence.
\end{proof}

\begin{theorem}\label{t:morita}
Let $\cc$ be an idempotent complete algebraic triangulated category. Assume that $\cc$ is generated by a set of objects $\{X_1,\ldots,X_r\}$, \ie $\cc=\thick(X_1,\ldots,X_r)$. Let $K$ be the direct product of $r$ copies of $k$. Then there is a strictly unital minimal $A_\infty$-algebra $B$ over $K$ such that as a graded algebra
\[
B= \bigoplus_{p\in\mathbb{Z}}\Hom_{\cc}(\bigoplus_{i=1}^r X_i,\Sigma^p \bigoplus_{i=1}^r X_i),
\]
and there is a triangle equivalence
\[
\per(B)\longrightarrow \cc
\]
taking  $e_iB$ ($1\leq i\leq r$) to $X_{i}$.
\end{theorem}
\begin{proof} This is a consequence of  \cite[Th\'eor\`eme 7.6.0.6]{Lefevre03}. The proof is similar to that of the preceding result.
\end{proof}

\section{Simple-minded collections and minimal positive $A_\infty$-algebras}
\label{s:smc-and-positive-A-infinite-alg}

In this section we recall the definition of simple-minded collections and study strictly unital minimal positive
$A_\infty$-algebras, which are closely related to simple-minded collections.

\subsection{Simple-minded collections}\label{ss:smc}
Let $\cc$ be a triangulated category with suspension functor $\Sigma$. A collection $\{X_1,\ldots,X_r\}$ of objects of $\cc$ is \emph{simple-minded} if
\begin{itemize}
\item[-] $\Hom_{\cc}(X_i,\Sigma^p X_j)=0,~~\forall~p<0 \text{ and } 1\leq i,j\leq r$,
\item[-] $\Hom_{\cc}(X_i,X_j)=0$ if $1\leq i\neq j\leq r$ and $\End_\cc(X_i)$ is a division $k$-algebra for all $1\leq i\leq r$,
\item[-] $\cc=\thick(X_1,\ldots,X_r)$.
\end{itemize}
Two simple-minded collections $\{X_1,\ldots,X_r\}$ and $\{X'_1,\ldots,X'_r\}$ are said to be \emph{isomorphic} if up to reordering we have $X_i\cong X'_i$ for all $1\leq i\leq r$.  A simple-minded collection $\{X_1,\ldots,X_r\}$ is said to be \emph{elementary} if $\End_\cc(X_i)\cong k$ for all $1\leq i\leq r$. If the field $k$ is algebraically closed, then every simple-minded collection in $\cc$ is elementary. In general $\cc$ may contain non-elementary simple-minded collections and it is not known whether any two simple-minded collections in $\cc$ have the same set of endomorphism algebras. It is the case when the two simple-minded collections are related by a mutation (\cite[Remark 7.7]{KoenigYang14}).

Let $A$ be a non-positive dg $k$-algebra with $H^0(A)$ being finite-dimensional over $k$. Then a complete set of pairwise non-isomorphic simple $H^0(A)$-modules, viewed as dg $A$-modules via the homomorphism $A\to H^0(A)$, form a simple-minded collection in $\cd_{fd}(A)$, see for example \cite[Theorem A.1 (c)]{BruestleYang13}.

\subsection{Strictly unital minimal positive $A_\infty$-algebras}

Fix $r\in\mathbb{N}$. Let $K$ be the direct product of $r$ copies of $k$. Let $e_1,\ldots,e_r$ be the standard basis of $K$ over $k$.

Let $A$ be a strictly unital minimal
$A_\infty$-algebra over $K$. We say that $A$ is \emph{positive} if
\begin{itemize}
\item[--] $A^p=0$ for all $p<0$,
\item[--] $A^0=K$ and the unit is the embedding $K=A^0\hookrightarrow A$.
\end{itemize}

It is clear that $e_1,\ldots,e_r$ are strict idempotents of $A$. So $e_1A,\ldots,e_rA$ are submodules of $A$.
Moreover, as an $A_\infty$-module $A=\bigoplus_{i=1}^r e_iA$.

\begin{lemma}\label{l:a-infinite-alg-and-simple-minded}
 \begin{itemize}
  \item[(a)] Let $A$ be a strictly unital minimal positive $A_\infty$-algebra over $K$.
  Then $\{e_1A,\ldots,e_rA\}$
is an elementary simple-minded collection in $\per(A)$.
  \item[(b)] Let $\cc$ be an idempotent complete algebraic triangulated category and let
  $\{X_1,\ldots,X_r\}$ be an elementary simple-minded collection in $\cc$. Then
there is a strictly unital minimal positive $A_\infty$-algebra $A$ over $K$
together with a triangle equivalence $\cc\rightarrow\per(A)$ which
takes $X_i$ ($1\leq i\leq r$) to $e_iA$.
 \end{itemize}
\end{lemma}
\begin{proof} (a) Put $P_i=e_iA$ ($1\leq i\leq r$). Then by Lemma~\ref{l:morphism-from-projective} we have
\begin{itemize}
\item[$\cdot$]
$\Hom(P_i,\Sigma^p P_j)=H^p(e_jAe_i)=0$ for $p<0$,
\item[$\cdot$]
$\Hom(P_i,P_j)=H^0(e_jAe_i)=\begin{cases} k & \text{ if } i=j, \\ 0
& \text{ otherwise.}\end{cases}$
\end{itemize}
Moreover, $P_1,\ldots,P_r$ generates $\per(A)$ since $A=P_1\oplus\ldots\oplus P_r$. Therefore $P_1,\ldots,P_r$ is an elementary simple-minded collection in
$\per(A)$.

(b) By Theorem~\ref{t:morita}, there is a strictly unital minimal $A_\infty$-algebra $A$ over $K$ such that as a graded algebra
\[
A=\bigoplus_{p\in\mathbb{Z}}\Hom_{\cc}(\bigoplus_{i=1}^r X_i,\Sigma^p \bigoplus_{i=1}^r X_i),
\]
and there is a triangle equivalence
\[
\cc \longrightarrow \per(A)
\]
taking  $X_{i}$ to $e_iA$, $1\leq i\leq r$. That $A$ is positive follows from the assumption that $\{X_1,\ldots,X_r\}$ is an elementary simple-minded collection.
\end{proof}

Let $A$ be a strictly unital minimal positive $A_\infty$-algebra over $K$. The projection $\varepsilon:A\to A^0=K$ makes $A$ an augmented $A_\infty$-algebra over $K$. We view $K$ as an $A_\infty$-module over $A$ via $\varepsilon$ and denote it by $S$. For $1\leq i\leq r$, we have a $1$-dimensional $A_\infty$-module $S_i=e_iA/e_i\ker(\varepsilon)$. Then $S=S_1\oplus\ldots\oplus S_r$. We call $S_1,\ldots,S_r$ the \emph{simple modules} over $A$. Let $M$ be a strictly unital $A_\infty$-module over $A$ which is concentrated in degree $0$. Then for all $m\in M$ we have $m_n^M(m\ten a_1\ten\cdots\ten a_{n-1})=0$ if  $a_1,\ldots,a_{n-1}$ are homogenous and one of them belongs to $\ker(\varepsilon)$. Indeed, we may assume that $m\in M^0$. If at least one of $a_1,\ldots,a_{n-1}$ belongs to $A^0$ and at least one of them belongs to $\ker(\varepsilon)$, then $m_n^M(m\ten a_1\ten\cdots\ten a_{n-1})=0$ because $M$ is strictly unital. If all $a_1,\ldots,a_{n-1}$ belong to $\ker(\varepsilon)$, then $m_n^M(m\ten a_1\ten\cdots\ten a_{n-1})$ is homogeneous of degree different from $0$ and has to be zero. So the $A_\infty$-module structure on $M$ factors through $\varepsilon$, so $M$ is the direct sum of copies of simple modules.

\begin{lemma}\label{l:simples-over-positive-A-infinity-algebras} Let $M$ be a strictly unital $A_\infty$-module over $A$. Fix $1\leq i\leq r$. If for $1\leq j\leq r$ and $p\in\mathbb{Z}$ we have
\[\Hom_{\cd(A)}(e_j A,\Sigma^p M)=\begin{cases} k & \text{ if } j=i \text{
and } p=0,\\
0 & \text{ otherwise,}
\end{cases}
\]
then $M$ is $A_\infty$-quasi-isomorphic to $S_i$.
\end{lemma}
\begin{proof}
By Lemma~\ref{l:morphism-from-projective}, we have
\[
H^p(Me_j)=\begin{cases} k & \text{ if } j=i \text{
and } p=0,\\
0 & \text{ otherwise.}
\end{cases}
\]
So by Proposition~\ref{p:minimal-model}, $M$ is $A_\infty$-quasi-isomorphic to a strictly unital $A_\infty$-module $M'$, which is $1$-dimensional and concentrated in degree $0$. Moreover $M'e_i=k$ and $M'e_j=0$ for $j\neq i$. So $M'$ is isomorphic to $S_i$.
\end{proof}

The first statement of the following lemma, phrased in terms of dg
algebras and dg modules, can be obtained by combining~\cite[Lemma 5.2]{KellerNicolas13} and \cite[Step 6 of the proof of Lemma 4.8]{AdachiMizunoYang17}.

\begin{lemma}\label{l:positive-a-infty-and-silting} The $A_\infty$-module $S$ is a silting object in
$\cd_{fd}(A)$\footnote{At first sight this statement may seem surprising. There are two typical situations: (1) the dg algebra/$A_\infty$-algebra is non-positive and the extensions between simple modules are in positive degrees, see the last paragraph of Section~\ref{ss:smc}; (2) the dg algebra/$A_\infty$-algebra is positive (this means that there are extensions in positive degrees between the `projective modules'! See Lemma 4.1) and the extensions between simple modules are in non-positive degrees.  Morally the first-order extensions between simple modules are dual to generators up to a degree shift. For example take the graded algebra $A=k[x]$ and the simple module $S=A/(x)$. Then apart from the scalar endomorphisms there are $1$-dimensional self-extensions of $S$, say with basis $y$. Then $|y|=1-|x|$. In particular, if the degree of $x$ changes from $-\infty$ to $+\infty$, then the degree of $y$ changes from $+\infty$ to $-\infty$  (as if there is a mirror at $\frac{1}{2}$).}. Moreover, if $A$ is Koszul as a graded algebra (see for example~\cite{BeilinsonGinzburgSoergel96} for a
definition), then
$S$ is a tilting object in $\cd_{fd}(A)$.
\end{lemma}

\begin{proof}
We first show that
$\Hom_{\cd(A)}(S,\Sigma^m S)=0$
for $m>0$.
Take the enveloping algebra $U$ of $A$  as in Theorem \ref{t:enveloping-algebra}. Then $U$ is an augmented dg algebra over $K$ and there is a strictly unital $A_\infty$-quasi-isomorphism $A\to U$ of augmented $A_\infty$-algebras over $K$. So $S$ can be viewed as a dg $U$-module and the $A_\infty$-structure on $S$ over $A$ factors through the $A_\infty$-quasi-isomorphism $A\to U$. By Theorem \ref{t:quasi-isomorphism-induce-derived-equivalence}, there is a triangle equivalence $\cd(U) \to \cd(A)$ which sends $S$ to $S$. So we only need to show that $\Hom_{\cd(U)}(S,\Sigma^m S)=0$ for $m >0$. We view $S$ as a graded module over the graded algebra $H^*(U)$, which is identified with $A$ viewed as a graded algebra. $S$ admits a projective resolution over the graded algebra $A$
\[
\xymatrix@C=1pc{\ldots\ar[r]& P^m\ar[r] & P^{m+1}\ar[r]&\ldots\ar[r]& P^{-1}\ar[r] & P^0}
\]
such that $P^m\in\Add(A\langle m'\rangle\mid m'\leq m)$, where $\langle 1 \rangle$ denotes the degree shift.   According to \cite[Theorem 3.1 (c)]{Keller94}, there is a dg module $P$ over $U$ which is quasi-isomorphic to $S$ and which admits a filtration
\[
    0=F_{-1} \subset F_{0} \subset \ldots \subset F_{p} \subset F_{p+1} \subset \ldots \subset P, ~~p \in \mathbb{N}
\]
such that
\begin{itemize}
\item[(F1)] $P$ is the union of the $F_{p}, ~~p \in \mathbb{N}$.
\item[(F2)] $\forall p \in \mathbb{N}$, the inclusion morphism $F_{p-1} \subset F_{p}$ splits in the category $\Grmod U$ of graded modules over $U$, which is considered as a graded algebra by forgetting the differential.
\item[(F3)] $\forall p\in\mathbb{N}$, $F_{p}/F_{p-1} \in \Add(\Sigma^{m} A\mid m \leq 0)$.
\end{itemize}

\smallskip
By (F1) and (F2), we have an isomorphism $P \cong \bigoplus_{p\geq 0} F_{p}/F_{p-1}$ in $\Grmod U$. So as a graded vector space $\cHom_{U}(P,S)=  \prod_{p \geq 0} \cHom_{U}(F_{p}/F_{p-1},S)$, which is concentrated in non-positive degrees by (F3), since $\cHom_U(\Sigma^m U,S)=\Sigma^{-m}S$. As a consequence, we obtain that for $m>0$
\[
\Hom_{\cd(U)}(S,\Sigma^m S) = \Hom_{\cd(U)}(P,\Sigma^m S) = H^{m}\cHom_{U}(P,S)=0.
\]

\smallskip
Next we show that $\cd_{fd}(A)=\thick(S)$. By Proposition~\ref{p:minimal-model}, it suffices to prove that
if a strictly unital minimal $A_\infty$-module $M$
over $A$ satisfies that $\mathrm{dim}(M):=\bigoplus_{m\in\mathbb{Z}}M^m$ is finite-dimensional, then $M\in \thick(S)$. Up to shift we may assume that
$M^m=0$ for all $m<0$ and $M^0\neq 0$. Define
$M^{>0}=\bigoplus_{m>0} M^m$.
Then $M^{>0}$ is an
submodule of $M$. Let $\iota: M^{>0}\to M$ be the embedding and form a triangle in $\cd(A)$
\[
\xymatrix{M^{>0}\ar[r]^\iota & M\ar[r] & \bar{M}\ar[r] & \Sigma M^{>0}.}
\]
Here $\bar{M}$ is assumed to be minimal. Looking at the long exact sequence of cohomologies, we see that $\bar{M}$ is
concentrated in degree $0$, and hence is a finite direct sum of copies of $S_1,\ldots,S_r$. Now by induction on $\dim(M)$ we
finish the proof.

\smallskip
Finally,  assume that $A$ is Koszul. Then the above resolution of $S$ can be chosen such that $P^m\in\Add(A\langle m\rangle)$. Consequently, $F_{p}/F_{p-1} \in \Add(U)$ and $\cHom_{U}(P,S)$ is concentrated in degree $0$. It follows that $\Hom_{\cd(U)}(S,\Sigma^m S)=0$ for $m \neq 0$.
\end{proof}

\section{Constructing silting objects from simples-minded collections}\label{s:construction}

In this section we will use Koszul duality to construct a silting object in the perfect derived category of a finite-dimensional non-positive dg algebra from a given simple-minded collection in the finite-dimensional derived category .
\medskip

Let $A$ be a
non-positive dg $k$-algebra whose total cohomology is finite-dimensional over $k$. 

\subsection{The standard simple-minded collection}
Note that $H^0(A)$ is a finite-dimensional algebra. Let $S_1,\ldots,S_r$ be a complete set of pairwise
non-isomorphic simple $H^0(A)$-modules and view them as dg $A$-modules via the homomorphism $A\to H^0(A)$. Recall that $\{S_1,\ldots,S_r\}$ is a simple-minded collection in $\cd_{fd}(A)$. We assume further that $\End_{H^0(A)}(S_i)\cong k$ for all $1\leq i\leq r$. Then $\{S_1,\ldots,S_r\}$ is an elementary simple-minded collection in $\cd_{fd}(A)$. We will see in the proof of Theorem~\ref{t:main-thm-for-negative-dg} that every elementary simple-minded collection in $\cd_{fd}(A)$ is of this form up to derived equivalence.

Since $\End_{\cd(A)}(A)=H^0(A)$, it follows that the functor $H^0=\Hom_{\cd(A)}(A,?)$ restricts to an equivalence $\add_{\cd(A)}(A)\stackrel{\simeq}{\to}\proj H^0(A)$. Therefore there are indecomposable objects $P_1,\ldots,P_r\in\add_{\cd(A)}(A)\subseteq\per(A)$ such that $H^0(P_1),\ldots,H^0(P_r)$ are projective covers of $S_1,\ldots,S_r$, respectively. In particular, there are positive integers $a_1,\ldots,a_r$ such that  $A\cong \bigoplus_{i=1}^r P_i^{\oplus a_i}$ in $\cd(A)$. Moreover, for $p\neq 0$, the space $\Hom_{\cd(A)}(P_i,\Sigma^p S_j)$ vanishes as it is a direct summand of $\Hom_{\cd(A)}(A,\Sigma^p S_j)=H^p(S_j)=0$. For $p=0$, consider the triangle
\begin{align}\label{eq:standard-triangle-for-projectives}
\sigma^{\leq -1}(P_i)\to P_i\to H^0(P_i)\to \Sigma\sigma^{\leq -1}(P_i),
\end{align}
where $\sigma^{\leq -1}$ is the standard truncation of complexes at degree $-1$.
Applying $\Hom_{\cd(A)}(?,S_j)$ to this triangle, we obtain a long exact sequence
\[
\Hom_{\cd(A)}(\Sigma\sigma^{\leq -1}(P_i),S_j)\to \Hom_{\cd(A)}(H^0(P_i),S_j)\to \Hom_{\cd(A)}(P_i,S_j)\to \Hom_{\cd(A)}(\sigma^{\leq -1}(P_i),S_j)
\]
By \cite[Theorem A.1 (b)]{BruestleYang13}, the two outer terms vanish and $\Hom_{\cd(A)}(H^0(P_i),S_j)$ is isomorphic to $\Hom_{H^0(A)}(H^0(P_i),S_j)$. So $\Hom_{\cd(A)}(P_i,S_j)\cong \Hom_{H^0(A)}(H^0(P_i),S_j)$ vanishes if $i\neq j$ and is isomorphic to $k$ if $i=j$. To sum up, we have 
\[\Hom_{\cd(A)}(P_i,\Sigma^p S_j)=\begin{cases} k & \text{ if } i=j \text{
and } p=0,\\
0 & \text{ otherwise.}
\end{cases}
\]
Further, the collection $\{S_1,\ldots,S_r\}$ is determined by this property. Namely, fix $1\leq j\leq r$ and let $M\in\cd(A)$ be such that 
\[\Hom_{\cd(A)}(P_i,\Sigma^p M)=\begin{cases} k & \text{ if } i=j \text{
and } p=0,\\
0 & \text{ otherwise,}
\end{cases}
\]
then $M\cong S_j$ in $\cd(A)$. Indeed, this property implies that $H^p(M)\cong \Hom_{\cd(A)}(A,\Sigma^p M)$ is trivial unless $p=0$, so $M$ is isomorphic in $\cd(A)$ to $H^0(M)$, which is a dg $A$-module via the homomorphism $A\to H^0(A)$. Moreover, applying $\Hom_{\cd(A)}(?,H^0(M))$ to the triangle \eqref{eq:standard-triangle-for-projectives} we obtain
\[\Hom_{H^0(A)}(H^0(P_i),H^0(M))\cong\Hom_{\cd(A)}(P_i,H^0(M))=\begin{cases} k & \text{ if } i=j,\\
0 & \text{ otherwise.}
\end{cases}
\]
It follows that $H^0(M)\cong S_j$ in $\mod H^0(A)$. Therefore $M\cong S_j$ in $\cd(A)$.

Let $I_i=\nu(P_i)$ for $1\leq i\leq r$. Then $D({}_AA)=\bigoplus_{i=1}^r I_i^{\oplus a_i}$ and by the Auslander--Reiten formula we have
\[\Hom_{\cd(A)}(S_i,\Sigma^p I_j)=\begin{cases} k & \text{ if } i=j \text{
and } p=0,\\
0 & \text{ otherwise.}
\end{cases}
\]

\newcommand{\rad}{\mathrm{rad}}
Let $K$ be the direct sum of $r$ copies of $k$ and let  $e_1,\ldots,e_r$ be the standard basis of $K$ over $k$. By Lemma~\ref{l:a-infinite-alg-and-simple-minded} (b), there is a strictly unital minimal
positive $A_\infty$-algebra $\cs$ over $K$ such that as a graded algebra
\[\cs=\bigoplus_{p\in\mathbb{Z}}\Hom_{\cd_{fd}(A)}(\bigoplus_{i=1}^r S_i,\Sigma^p \bigoplus_{i=1}^r S_j)
\]
and there is a triangle equivalence
\[
\xymatrix{\Phi:\cd_{fd}(A) \ar[r] & \per(\cs)}
\]
taking $S_i$ ($1\leq i\leq r$) to
$e_i\cs$.
Therefore we have
\[
\Hom_{\cd(\cs)}(e_i \cs,\Sigma^p \Phi(I_j))=\begin{cases} k & \text{ if } i=j \text{
and } p=0,\\
0 & \text{ otherwise.}
\end{cases}
\]
By Lemma~\ref{l:simples-over-positive-A-infinity-algebras}, $\Phi(I_1),\ldots,\Phi(I_r)$ are, up to $A_\infty$-quasi-isomorphism, precisely the
simple modules over $\cs$. In other words, the
equivalence $\Phi$ restricts to a triangle equivalence
\[\xymatrix{\Phi|:\thick_{\cd(A)}(D({}_AA))=\thick_{\cd(A)}(I_1,\ldots,I_r)\ar[r] & \thick_{\cd(\cs)}(\Phi(I_1),\ldots,\Phi(I_r))=\cd_{fd}(\cs),}\]
where the last equality follows from
Lemma~\ref{l:positive-a-infty-and-silting}. It follows that $\cd_{fd}(\cs)\subseteq \per(\cs)$.

\subsection{Construction of the silting object}

Let $\{X_1,\ldots,X_r\}\subseteq\mathcal{D}_{fd}(A)$ be an elementary
simple-minded collection.   Let $Y_i=\Phi(X_i)$ for $1\leq i\leq r$. Then $\{Y_1,\ldots,Y_r\}$ is an elementary simple-minded collection in $\per(\cs)$. By Theorem~\ref{t:morita-for-der-cat}, there is a strictly unital minimal
positive $A_\infty$-algebra $\cx$ over $K$ such that as a graded algebra
\[\cx=\bigoplus_{p\in\mathbb{Z}}\Hom_{\cd_{fd}(A)}(\bigoplus_{i=1}^r Y_i,\Sigma^p \bigoplus_{i=1}^r Y_i)
\]
and there is  a triangle equivalence
\[
\xymatrix{\tilde{\Psi}:\cd(\cs) \ar[r] & \cd
(\cx)}
\]
taking $Y_i$ ($1\leq i\leq r$) to
$e_i\cx$. By Lemma~\ref{l:der-equiv-restricts-to-per-and-fd}, $\tilde{\Psi}$ restricts to
triangle equivalences
\[\Psi:\per(\cs)\longrightarrow \per(\cx),\]
\[\Psi|:\cd_{fd}(\cs)\longrightarrow\cd_{fd}(\cx).\]
This implies that $\cd_{fd}(\cx)\subseteq \per(\cx)$. So we have the following commutative diagram
\[
\xymatrix{\cd_{fd}(A) \ar[r]^{\Phi} & \per(\cs)\ar[r]^{\Psi} & \per(\cx)\\
\thick(D({}_AA))\ar@{^{(}->}[u] \ar[r]^{\Phi|} & \cd_{fd}(\cs)\ar[r]^{\Psi|} \ar@{^{(}->}[u]& \cd_{fd}(\cx)\ar@{^{(}->}[u]}
\]

Let $R_1,\ldots,R_r$ be the
 simple modules over $\cx$, and let $T_1,\ldots,T_r$ be
their images under a quasi-inverse of the equivalence
$(\Psi\circ \Phi)|$. Put $T=\bigoplus_{i=1}^r T_i$.

\begin{proposition}\label{p:new-construction}
\begin{itemize}
\item[(a)] $T$ is a silting object of $\thick(D({}_AA))$.
\item[(b)] For $1\leq i,j\leq r$, and
$p\in\mathbb{Z}$,
\[\Hom_{\cd_{fd}(A)}(X_j,\Sigma^p T_i)=\begin{cases} k & \text{if }i=j\text{ and } p=0,\\
                                           0 & \text{otherwise}.
                                           \end{cases}\]
\item[(c)] $\nu^{-1}T$ is a silting object of $\per(A)$.
\item[(d)] For $1\leq i,j\leq r$, and
$m\in\mathbb{Z}$,
\[\Hom_{\cd_{fd}(A)}(\nu^{-1}T_i,\Sigma^p X_j)=\begin{cases} k & \text{if }i=j\text{ and } p=0,\\
                                           0 & \text{otherwise}.
                                           \end{cases}\]

\end{itemize}
\end{proposition}
\begin{proof} (a) This is because $R_1\oplus\ldots\oplus R_r$ is a silting object of $\cd_{fd}(\cx)$ (Lemma~\ref{l:positive-a-infty-and-silting}) and $(\Psi\circ\Phi)|$ is a triangle equivalence.

(b) By Lemma~\ref{l:morphism-from-projective} we have
\[\Hom(e_j\cx,\Sigma^p R_i)=\begin{cases} k & \text{ if } i=j \text{
and } p=0,\\
0 & \text{ otherwise.}
\end{cases}
\]
The desired formula follows immediately because $\Psi\circ\Phi$ is a triangle equivalence.

(c) follows from (a) because $\nu:\per(A)\to \thick(D({}_AA))$ is a triangle equivalence.

(d) follows from (b) and the Auslander--Reiten formula.
\end{proof}

If $A$ is a finite-dimensional elementary (ordinary) $k$-algebra, then this is \cite[Lemmas 5.6, 5.7 and 5.8 and Proposition 5.9]{KoenigYang14}, up to the hypothesis that if $\cd^b(\mod A)=\cd_{fd}(A)$ has an elementary simple-minded collection, then all simple-minded collections in $\cd^b(\mod A)$ are elementary.
Compared with Rickard's construction \cite{Rickard02} used in \cite{KoenigYang14}, our new approach has the disadvantage that it may fail for non-elementary simple-minded collections, but it also has some advantages.  By  Lemma~\ref{l:positive-a-infty-and-silting}, we obtain a sufficient condition on
$\nu^{-1}T$ being a tilting object.

\begin{corollary}\label{c:sufficient-condition-for-T-to-be-tilting}
If $\cx$ as a graded algebra is
Koszul, then $\nu^{-1}T$ is a tilting object of $\per(A)$.
\end{corollary}

By \cite{SuHao16}, if $\cx$ as a graded algebra is isomorphic to $kQ$, where $Q$ is a graded quiver with all arrows in positive degrees, then $\cd_{fd}(A)$ is triangle equivalent to $\per(kQ)$, where $kQ$ is considered as a dg algebra with trivial differential. If all arrows of $Q$ are in degree $1$, then it is known that $\per(kQ)$ is triangle equivalent to the radical-square-zero algebra $R$ associated to the opposite quiver $Q^{op}$, considered as an ungraded quiver, see for example \cite[Theorem 2.5]{ChenYang15}. Consequently,  $A$ is derived equivalent to $R$.

\medskip

Furthermore, the dg endomorphism algebra and the truncated dg endomorphism algebra $\tilde{\Gamma}$ of $\nu^{-1}T$ (see Section~\ref{ss:non-positive-dg-alg}) are
Koszul dual to the $A_\infty$-algebra $\cx$. So they can
be obtained, up to quasi-equivalence (in the sense of \cite[Section 7]{Keller94}), as the dual bar construction
of $\cx$:
the complete tensor algebra $B^{\#}\cs=\widehat{T}_K(D(\cx^{>0}[1]))$ of $D(\cx^{>0}[1])$ over $K$  (see~\cite{LuPalmieriWuZhang08,KalckYang16}), where $\cx^{>0}=\bigoplus_{p>0}\cx^p$. In other words, up to
quasi-equivalence the following diagram is commutative
\[\xymatrix{\{X_1,\ldots,X_r\}\ar@{|->}[r]\ar@{|->}[d] & \nu^{-1}T\ar@{|->}[d]\\
\cx\ar@{|->}[r]^{dual~~ bar}_{construction} & \tilde{\Gamma}.} \] In
general the $A_\infty$-structure on $\cx$ is hard to compute.
However, sometimes it is easy to obtain the $A_\infty$-structure on
the truncated part $\cx^{[0,2]}$, the $A_\infty$-algebra obtained from $\cx$ by modulo the elements of degree $\geq 3$.  We have
\[
\End_{\cd(A)}(\nu^{-1}T)=H^0(\cEnd_A(\nu^{-1}T))=H^0(B^{\#}\cx)=H^0(B^{\#}\cx^{[0,2]}).
\]
The quiver $Q$ of $\End_{\cd(A)}(\nu^{-1}T)$ is determined by the $K$-bimodule structure on $\cx^1$. Precisely, the set of vertices of $Q$ is $\{1,\ldots,r\}$, and the number of arrows from $i$ to $j$ is the dimension of $e_i\cx^1e_j$ over $k$. The relations of $\End_{\cd(A)}(\nu^{-1}T)$ are `dual' to the restrictions $m_n:(\cx^1)^{\ten_K n}\to\cx^2$ of the multiplications of $\cx$.

\begin{corollary}
If $\cx^1=0$, then as an algebra $\End_{\cd(A)}(\nu^{-1}T)$ is isomorphic to $K$.
\end{corollary}

The following example is taken from~\cite{Al-Nofayee09}. Let $A$ be the algebra given by the quiver
$\xymatrix{1\ar@<.7ex>[r]^{\alpha}&2\ar@<.7ex>[l]^{\beta}}$ with
relations $\alpha\beta=0=\beta\alpha$. Take $X_1=P_1$ and $X_2=\Sigma^{-1}S_1$.
The $\cx$ as a graded algebra is the path algebra of the graded quiver
\[\xymatrix{1\ar[r]^{\gamma} & 2\ar@(ur,dr)^\delta}\]
where $\gamma$ is of degree $1$ and $\delta$ is of degree $2$.
Simply because of lack of morphisms to multiply with, the
$A_\infty$-structure on $\cx^{[0,2]}$ is trivial.
The dual bar construction shows that $\End_{\cd(A)}(\nu^{-1}T)$ is
the path algebra of the ungraded quiver $\xymatrix{1& 2\ar[l]}$.

\subsection{Main results}

Now we state our main result, which is a consequence of Proposition~\ref{p:new-construction}. 

\begin{theorem}\label{t:from-smc-to-silting}
For an elementary simple-minded collection $\{X_1,\ldots,X_r\}$ of $\cd_{fd}(A)$, there exists a unique (up to isomorphism) silting object $M=M_1\oplus\ldots\oplus M_r$ of $\per(A)$ such that
for $1\leq i,j \leq r$ and
$p\in\mathbb{Z}$
\[\Hom_{\cd_{fd}(A)}(M_i,\Sigma^p X_j)=\begin{cases} k & \text{if }i=j\text{ and } p=0,\\
                                           0 & \text{otherwise.}
                                           \end{cases}\]
\end{theorem}

\begin{proof} The existence of $M$ follows from Proposition~\ref{p:new-construction} (c)(d): take $M=\nu^{-1}T$. 
Let $N=N_1\oplus\ldots\oplus N_r$ be an object of $\per(A)$ such that
for $1\leq i,j\leq r$ and
$p\in\mathbb{Z}$
\[\Hom_{\cd_{fd}(A)}(N_i,\Sigma^p X_j)=\begin{cases} k & \text{if }i=j\text{ and } p=0,\\
                                           0 & \text{otherwise.}
                                           \end{cases}\]
Then by the Auslander--Reiten formula we have
\[\Hom_{\cd_{fd}(A)}(\Sigma^p X_j,\nu N_i)=\begin{cases} k & \text{if }i=j\text{ and } p=0,\\
                                           0 & \text{otherwise.}
                                           \end{cases}\]
Applying the triangle equivalence $\Psi\circ\Phi$ we obtain      
\[\Hom_{\cd_{fd}(A)}(\Sigma^p e_j\cx,\Psi\circ\Phi\circ\nu(N_i))=\begin{cases} k & \text{if }i=j\text{ and } p=0,\\
                                           0 & \text{otherwise.}
                                           \end{cases}\]
Therefore by Lemma~\ref{l:simples-over-positive-A-infinity-algebras} we have $\Psi\circ\Phi\circ\nu(N_i)\cong R_i$. So $N_i\cong \nu^{-1}\circ (\Psi\circ\Psi)^{-1}(R_i)=M_i$, showing the uniqueness of $M$.
\end{proof}

As a consequence, we obtain the following theorem, generalising \cite[Theorems 6.1 and 7.12]{KoenigYang14}. The assumption that $k$ is algebraically closed is required to ensure that all simple-minded collections are elementary, which is used only in the construction of the map from simple-minded collections to silting objects.

\begin{theorem}\label{t:main-thm-for-negative-dg}
Assume that $k$ is algebraically closed. Then there are one-to-one correspondences which commute with mutations and which preserve partial orders
between
\begin{itemize}
 \item[(1)] equivalence classes of silting objects in $\per(A)$,
 \item[(2)] isomorphism classes of simple-minded collections in $\cd_{fd}(A)$,
 \item[(3)] bounded $t$-structures of $\cd_{fd}(A)$ with length heart,
 \item[(4)] bounded co-$t$-structures of $\per(A)$.
\end{itemize}
\end{theorem}
\begin{proof}
The proof is the same as that of \cite[Theorems 6.1 and 7.12]{KoenigYang14}. Here we only give the definition of some of the correspondences, which are compatible with each other.

{\it From silting objects to simple-minded collections:} Let $M$ be a basic silting object in $\per(A)$.We may assume that $M$ is $K$-projective. Let $\tilde{\Gamma}$ be the truncated dg endomorphism algebra of $M$ (see Section~\ref{ss:non-positive-dg-alg}).  Then  $\tilde{\Gamma}$ is non-positive and has finite-dimensional total cohomology; moreover, there is a triangle equivalence $\cd(\tilde{\Gamma})\rightarrow \cd(A)$ taking $\tilde{\Gamma}$ to $M$. The simple-minded collection $\{X_1,\ldots,X_r\}$ corresponding to $M$ is the image of a complete collection of pairwise non-isomorphic simple $H^0(\tilde{\Gamma})$-modules (viewed as dg $\tilde{\Gamma}$-modules) under this equivalence. It is the unique collection (up to isomorphism) in $\cd_{fd}(A)$ satisfying for  $1\leq i,j\leq r$ and  $p\in\mathbb{Z}$
\begin{align*}
\Hom(M_i,\Sigma^p X_j)=\begin{cases} k & \text{ if } i=j \text{ and } p=0,\\ 0 & \text{ otherwise}.\end{cases}
\end{align*}

{\it From silting objects to t-structures:} Let $M$ be a silting object in $\per(A)$. The corresponding $t$-structure $(\cd^{\leq 0},\cd^{\geq 0})$ on $\cd_{fd}(A)$ is defined as
\begin{align*}
\cd^{\leq 0}&=\{X\in\cd_{fd}(A)\mid \Hom(M,\Sigma^p X)=0 \text{ for } p>0\},\\
\cd^{\geq 0}&=\{X\in\cd_{fd}(A)\mid\Hom(M,\Sigma^p X)=0\text{ for } p<0\}.
\end{align*}
This is the image of the standard $t$-structure (\cite[Theorem A.1]{BruestleYang13}) on $\cd_{fd}(\tilde{\Gamma})$ under the triangle equivalence $\cd_{fd}(\tilde{\Gamma})\to\cd_{fd}(A)$ (which is restricted from the triangle equivalence $\cd(\tilde{\Gamma})\to\cd(A)$). The heart of this t-structure is equivalent to $\mod \End(M)$.

{\it From silting objects to co-t-structures:} Let $M$ be a silting object in $\per(A)$. The corresponding co-$t$-structure $(\cp_{\geq 0},\cp_{\leq 0})$ on $\per(A)$ is defined as (see \cite[Section 3.4]{KoenigYang14})
\begin{align*}
\cp_{\geq 0} & = \text{the smallest full subcategory of $\per(A)$ which contains $\{\Sigma^p M\mid p\leq 0\}$}\\
& \text{and which is closed under taking extensions and direct summands},\\
\cp_{\leq 0} & = \text{the smallest full subcategory of $\per(A)$ which contains $\{\Sigma^p M\mid p\geq 0\}$}\\
& \text{and which is closed under taking extensions and direct summands}.
\end{align*}

{\it From t-structures to simple-minded collections:} Let $(\cd^{\leq 0},\cd^{\geq 0})$ be a bounded $t$-structure on $\cd_{fd}(A)$ with length heart. The corresponding simple-minded collection is a complete collection of pairwise non-isomorphic simple objects of the heart $\cd^{\leq 0}\cap\cd^{\geq 0}$ (see \cite[Section 3.3]{KoenigYang14}).

{\it From simple-minded collections to silting objects:} Let $\{X_1,\ldots,X_r\}$ be a simple-minded collection of $\cd_{fd}(A)$. It is elementary since the base field $k$ is algebraically closed. The corresponding silting object is the $M$ as in Theorem~\ref{t:from-smc-to-silting}.

{\it From co-t-structures to silting objects:} Let $(\cp_{\geq 0},\cp_{\leq 0})$ be a bounded co-$t$-structure on $\per(A)$. The corresponding silting object $M$ of $\per(A)$ is an additive generator of the co-heart, \ie $\add(M)=\cp_{\geq 0}\cap\cp_{\leq 0}$ (see \cite[Sections 3.1 and 3.4]{KoenigYang14}).
\end{proof}

\def\cprime{$'$}
\providecommand{\bysame}{\leavevmode\hbox to3em{\hrulefill}\thinspace}
\providecommand{\MR}{\relax\ifhmode\unskip\space\fi MR }
\providecommand{\MRhref}[2]{%
  \href{http://www.ams.org/mathscinet-getitem?mr=#1}{#2}
}
\providecommand{\href}[2]{#2}

\end{document}